\documentclass[11pt,notitlepage,twoside,a4paper]{amsart}
 \usepackage{amsfonts}

\usepackage{amsmath,amssymb,enumerate}

\usepackage{epsfig,fancyhdr,color}

\usepackage{psfrag}
\usepackage{amssymb}
\usepackage{amsmath,amsthm}
\usepackage{latexsym}
\usepackage{amscd}
\usepackage{psfrag}
\usepackage{graphicx}
\usepackage[latin1]{inputenc}
\usepackage[all]{xy}
\usepackage[mathcal]{eucal}
\usepackage{amsthm}

\definecolor{NoteColor}{rgb}{1,0,0}

\renewcommand{\textsc}{\textcolor{red}}

\newtheorem{theorem}{\rm\bf Theorem}[section]

\newtheorem{lemma}[theorem]{\rm\bf Lemma}

\newtheorem*{theorem 1}{\rm\bf Proposition 1}
\newtheorem*{theorem 2}{\rm\bf Proposition 2}

\theoremstyle{definition}

\theoremstyle{remark}

\def\interieur#1{\mathord{\mathop{\kern 0pt #1}\limits^\circ}}

\title[Thurston's metric]{Thurston's metric on Teichm\"uller space and isomorphisms between Fuchsian groups}

\author{A. Papadopoulos}
 \address{Athanase Papadopoulos,  Universit{\'e} de Strasbourg and CNRS,
7 rue Ren\'e Descartes,
 67084 Strasbourg Cedex, France}
\email{athanase.papadopoulos@math.unistra.fr}

\author{W.  Su}
\address{Weixu Su, Department of Mathematics, Fudan University, 200433, Shanghai, P. R. China}
\email{suweixu@gmail.com}

\begin{document}

\maketitle

\begin{abstract}
The aim of this paper is to relate Thurston's metric on Teichm\"uller space to several ideas initiated by T. Sorvali on isomorphisms between Fuchsian groups. In particular, this will give a new formula for Thurston's asymmetric metric for surfaces with punctures.
We also update some results of Sorvali on boundary isomorphisms of Fuchsian groups.
\end{abstract}
\bigskip

\noindent AMS Mathematics Subject Classification:   32G15 ; 30F30 ; 30F60.
\medskip

\noindent Keywords:  Thurston's metric, Teichm\"uller space, Fuchsian group, boundary mapping, translation length.
\medskip

\section{Introduction}\label{intro}

Let $S=S_{g,n}$ be an oriented surface of genus $g$ with $n$ punctures.
The Teichm\"uller space
$\mathcal{T}_{g,n}$ of $S$ is the space of equivalence classes of complete hyperbolic structures of finite area on $S$, where
 two hyperbolic structures $X$ and $Y$ on $S$ are \emph{equivalent} if there exists an isometry $h:X \to Y$
homotopic to the identity on $S$.

Thurston \cite{Thurston}  defined an asymmetric metric $d_L$, which we call, for brevity, the \emph{Thurston metric}, on  $\mathcal{T}_{g,n}$
by setting
\begin{equation} \label{equation:lip}
d_L(X,Y)=\inf_{f} \log L_f(X,Y),
\end{equation}
where \[L_f(X,Y)=\sup_{x,y\in S, x\not= y}\frac{d_Y(x,y)}{d_X(x,y)}\] is the Lipschitz constant of a homeomorphism $f:X \to Y$ homotopic to the identity map of $S$. In the same paper, Thurston proved that  there is a (non-necessarily unique) extremal Lipschitz homeomorphism that realizes the infimum in $(\ref{equation:lip} )$, and  that
$$d_L(X,Y)=\log  K(X, Y),$$
where
$$ K(X, Y)=\sup_{\gamma\in\mathcal{S}}\frac{\ell_Y(\gamma)}{\ell_X(\gamma)},$$
where $\ell_X(\gamma)$ denotes the hyperbolic length of $\gamma$ in $X$ and
$\mathcal{S}$ is the set of homotopy classes of essential simple closed curves on  $S$. Several geometrical aspects of Thurston's metric, such as the description of a distinguished class of geodesics (called stretch lines) and the description of the structure of its Finsler norm unit ball, were  studied by
 Thurston in \cite{Thurston}. Thurston's metric is also related to Thurston's compactification of
Teichm\"uller space, see \cite{Papadopoulos} and \cite{wal}.

In this paper,  we consider hyperbolic surfaces with at least one puncture. The existence of punctures is equivalent to the fact that the Fuchsian groups that represent the hyperbolic surfaces contain parabolic transformations. We relate Thurston's metric to a work of Sorvali, which was done more than ten years before the appearance of Thurston's preprint \cite{Thurston}.  Sorvali's work concerns a symmetrization of Thurston's metric, namely, the so-called  \emph{length-spectrum} metric, but part of his theory may be used for a description of Thurston's metric.
The length-spectrum metric $d_{ls}$ on $\mathcal{T}_{g,n}$ is defined,
for $X$ and $Y$ in $\mathcal{T}_{g,n}$, by

\[
d_{ls}(X,Y)=\log \sup_{\gamma\in \mathcal{S}} \{ \frac{\ell_Y(\gamma)}{\ell_X(\gamma)}, \frac{\ell_X(\gamma)}{\ell_Y(\gamma)}  \}
\]
which, by Thurston's result mentioned above, is equal to
\[ \max \{  d_L(X,Y), d_L(Y,X)\}.\]
For surfaces with punctures, the work of Sorvali gives a formula for $d_{ls}(X,Y)$ in terms of the \emph{translation vector} (defined below) of the parabolic transformations.  Using these ideas, we obtain a new formula for the Lipschitz distance between two hyperbolic structures $X$ and $Y$ in terms of translation vectors of parabolic transformations corresponding to punctures of $X$ and $Y$.

Sorvali \cite{Sorvali73} also related the length-spectrum distance to
the H\"older continuity of  the \emph{boundary mappings} (see Section \ref{boundary} for the definition). His results are also interesting for surfaces of infinite type. In a previous paper \cite{ALPS}, we have observed that for surfaces of infinite type, the definition of the associated Teichm\"uller space depends on the choice of a base-point of that space and  on the choice of a metric on that space which induces the topology. We used the name \emph{quasiconformal Teichm\"uller space} for a Teichm\"uller space equipped with the Teichm\"uller metric, and \emph{length-spectrum Teichm\"uller space} for a Teichm\"uller space equipped with the length-spectrum metric. The spaces are usually different, even if the base surfaces are the same. In particular, we saw that  there exists a hyperbolic surface $R$ of infinite type such that the quasiconformal Teichm\"uller space $\mathcal{T}_{qc}(R)$ is a proper subset of the length-spectrum Teichm\"uller space $\mathcal{T}_{ls}(R)$. Combining this with the result of Sorvali \cite{Sorvali73}, we obtain a class of  examples of homeomorphisms on $\mathbb{R}\cup \{\infty\}$ which are H\"older continuous but not quasisymmetric.

\section{Isomorphism between Fuchsian groups}\label{iso}

Let $\mathbb{H}$ be the upper half-plane endowed with the Poincar\'e metric. The ideal boundary of $\mathbb{H}$ can be identified with $\overline{\mathbb{R}}=\mathbb{R}\cup \{\infty\}$.
 A \emph{Fuchsian group } $\Gamma$ is a subgroup of $\mathrm{PSL}(2, \mathbb{R})$ which acts properly discontinuously and freely on $\mathbb{H}$.

  In this section, we consider Fuchsian groups that are not cyclic and which contain the parabolic transformation $g_0: z \mapsto z+1$.  We do not assume that they are finitely generated and, consequently, the surfaces we consider might be of infinite type.

  To a hyperbolic isometry $g$ of the upper half-plane  is associated a \emph{multiplier} $\lambda (g)>1$ whose logarithm is the \emph{translation length} along the invariant geodesic of $g$. To a parabolic isometry, we can associate a \emph{translation length along a horocycle}. The latter association is not canonical (it needs some normalization), but it will turn out to be useful. We shall make this precise now.

  An isomorphism $j: \Gamma \to \Gamma'$ between Fuchsian groups is called \emph{type-preserving} if it maps parabolic elements to parabolic elements and hyperbolic elements to hyperbolic elements. This is equivalent to saying that $j$ and $j^{-1}$ both preserve parabolic elements. Note that if there exists a homeomorphism $f: \overline{\mathbb{R}}\to \overline{\mathbb{R}}$ such that
$$f\circ g=j(g)\circ f$$
for all $g\in \Gamma$, then $j$ is type-preserving. Indeed, in this case, $f\circ g \circ f^{-1}=j(g)$, and if $a\in \overline{\mathbb{R}}$ is a fixed points of $g$, then $f(a)$ is a  fixed points of $j(g)$.

In this section, we shall only consider type-preserving isomorphisms $j: \Gamma \to \Gamma'$.

For a hyperbolic transformation $g$, we denote its attracting and repelling fixed points by $P(g)$
and $N(g)$ respectively. The element $g$ is uniquely determined by $P(g), N(g)$ and by the \emph{multiplier } $\lambda(g)>1$, defined by the fact that $\log\lambda(g)$ is the translation length along the invariant geodesic of $g$. From the definition, we see that $\lambda(g)$ is a conjugacy invariant. Up to  conjugation by an element of $\mathrm{PSL}(2,\mathbb{R})$, $g$ is represented by the transformation $z \mapsto \lambda(g) z$. Since the eigenvalues of a M\"obuis transformation  $z\mapsto \lambda(g) z$ are $\lambda(g)^{1/2}$ and $\lambda(g)^{-1/2}$, its trace is
$$\mathrm{tr}(g)= \lambda(g)^{1/2}+ \lambda(g)^{-1/2}.$$

If $g$ is parabolic, we set $\lambda(g)=1$ and we denote the unique fixed point of $g$ by $P(g)$.

Given an isomorphism $j: \Gamma \to \Gamma'$, we define $\delta_L(j)\in [1, \infty]$ by
$$\delta_L(j)=\inf \{a\geq 1: \lambda(j(g))\leq  \lambda(g)^a, \  \forall \ g\in \Gamma\}.$$
This definition of $\delta_L(j)$ is a nonsymmetric version  of a definition made by Sorvali in \cite{Sorvali73}.

\begin{lemma}\label{lemma:tr}
If $1\leq s\leq \infty$ is the smallest number such that for all $ g\in \Gamma$,
$\mathrm{tr}(j(g))\leq \mathrm{tr}(g)^s,$
then $s=\delta_L(j)$.
\end{lemma}
\begin{proof}
Suppose that $\mathrm{tr}(j(g))\leq  \mathrm{tr}(g)^a$ for all $ g\in \Gamma$. For any $g\in \Gamma$, we let $\lambda=\lambda(g), \lambda'=\lambda(j(g))$.   Since $j$ is an isomorphism,  $j(g^n)=j(g)^n$, $n=1, 2, \cdots$. Since $\mathrm{tr}(j(g)^n)\leq \mathrm{tr}(g^n)^a$, we have
$$(\lambda')^{n/2}+ (\lambda')^{-n/2}\leq (\lambda^{n/2}+ \lambda^{-n/2})^a,$$
and then
$$(\lambda')^n\leq (\lambda')^{n}+ (\lambda')^{-n}+2\leq (\lambda^{n}+ \lambda^{-n}+2)^a.$$
For some sufficiently large integer $n_0$, the right hand side of the above inequality is less than $(4\lambda^n)^a$ for all $n\geq n_0$. Therefore,
$$\lambda' \leq (4^{1/n}\lambda)^a,  \  \forall \ n\geq n_0. $$
By letting $n\to \infty$, we get $\lambda'\leq \lambda^a$.

Conversely, suppose that $\lambda'\leq \lambda^a$. Then
$$\mathrm{tr}(j(g))= (\lambda')^{1/2}+ (\lambda')^{-1/2}\leq \lambda^{a/2}+ \lambda^{-a/2}\leq (\lambda^{1/2}+ \lambda^{-1/2})^a=\mathrm{tr}(g)^a.$$
\end{proof}

Let $g$ be a parabolic transformation. We shall associate to it a real number $\omega(g)$.

 If $P(g)=\infty$ and $g(z)=z+c$, then we define $\omega(g)=c$.
The value $\omega(g)$ is a real (maybe positive or negative) number uniquely determined by $g$. The absolute value $|\omega(g)|$ is expressed (like the value $\lambda(g)$ associated to a hyperbolic element $g$) as a translation length, namely, the  length of the horizontal segment (horocycle) joining the complex numbers $i$ and $ i+ \omega(g)$.
If $g $ is parabolic with $P(g)\neq \infty$, then we define $\omega(g)$ to be the real number determined by the following equality:
$$\frac{1}{g(z)-P(g)}= \frac{1}{z-P(g)}+ \omega(g).$$
There is a unique horocycle  through $P(g)$ and $P(g)+i$ which is invariant by the action of $g$. For any point $z$ on the horocycle, $|\omega(g) |$ is the non-euclidean length of the horocycle arc connecting $z$ and $g(z)$.

Thus, when $g$ is parabolic, and in both cases $P(g)=\infty$ and $P(g)\not=\infty$, $\vert w(g)\vert $ is a translation length along a horocycle.

Following Sorvali \cite{Sorvali73}, we call $\omega(g)$ the \emph{translation vector} of $g$.

Note that unlike the value $\lambda(g)$ associated to a hyperbolic element $g$, the values $\omega(g)$ and  $\vert w(g)\vert$ are not conjugacy invariants in the case $g$ is parabolic. To see this, consider the parabolic element $\displaystyle g_(z)= \frac{z}{1+cz}$ where $c \neq 0$ is real and let $h$ be a transformation $h(z)=\lambda z$ with $\lambda \in (0,\infty[$. Then a computation gives $\omega(g)=c$ while $\omega(h^{-1} \circ g \circ h )=c\lambda$. The fact that $\omega(g)$ is not a conjugacy invariant  can also be seen from the following relation:

From now on and except in Theorem \ref{th:S} below, we shall only consider type-preserving isomorphisms $j: \Gamma \to \Gamma'$ that fix $g_0$.

\begin{lemma} \label{lemma:square}
For any parabolic element $h$ in $ \mathrm{PSL}(2,\mathbb{R})$ with $P(h)\neq \infty $, we have $$|\omega(h^{-1} \circ g_0 \circ h ) |=|\omega(h) |^2.$$
\end{lemma}
\begin{proof} The proof is contained in the proof of Theorem 3 of \cite{Sorvali73}. We reproduce it here for the convenience of the reader.

We have
$$\frac{1}{h(z)-P(h)}= \frac{1}{z-P(h)}+ \omega(h)$$
or, equivalently,
$$h(z)=\frac{\left(1+\omega(h)P(h)\right)z-\omega(h)P(h)^2}{\omega(h)z+1-\omega(h)P(h)}.$$
A computation shows that
\begin{equation}\label{equation:conjugate}
(h^{-1} \circ g_0 \circ h) (z)=\frac{\left( 1+\omega(h)-\omega(h)^2P(h) \right) z-\left( 1-\omega(h)P(h) \right)^2}{-\omega(h)^2 z+1-\omega(h)+\omega(h)^2P(h)}.
\end{equation}
Since the fixed point of $h^{-1} \circ g_0 \circ h$  is $h^{-1} (P(g_0))=h^{-1}(\infty)$, we have
$$P(h^{-1} \circ g_0 \circ h)=\displaystyle\frac{\omega(h)P(h)-1}{\omega(h)}.$$
  It is easy to show that
$(\ref{equation:conjugate})$ is equivalent to
$$\frac{1}{(h^{-1} \circ g_0 \circ h)(z)-P(h^{-1} \circ g_0 \circ h)}= \frac{1}{z-P(h^{-1} \circ g_0 \circ h)}- \omega(h)^2.$$
It follows that $\omega(h^{-1} \circ g_0 \circ h)=-\omega(h)^2$.

\end{proof}

In what follows, we consider an isomorphism $j: \Gamma \to \Gamma'$ satisfying $\delta_L(j)< \infty$. Such a hypothesis is satisfied for example if $j$ is induced by a homeomorphism between $\mathbb{H}/ \Gamma$ and $\mathbb{H}/ \Gamma'$,  where $\mathbb{H}/ \Gamma$ and
$\mathbb{H}/ \Gamma'$ are surfaces of finite type.

The following is an asymmetric version of Theorem 3 in Sorvali's paper \cite{Sorvali73}.
\begin{lemma} \label{lemma:para}
Let $j: \Gamma \to \Gamma'$ be an isomorphism such that $s=\delta_L(j)< \infty$.
Then $|\omega(j(g)) |\leq  |\omega(g) |^s$ for all parabolic transformations $g \in \Gamma$.
\end{lemma}
\begin{proof}
We follow the proof of Theorem 10 of \cite{Sorvali74}.
Let $g\neq g_0$ be a fixed parabolic transformation of $\Gamma$. Since $\Gamma$ is discrete,
$P(g)\neq \infty$.

Let  $g_1=g^{-1} \circ g_0 \circ g$ and inductively $g_n=g_{n-1}^{-1} \circ g_0 \circ g_{n-1}$ for every $n\geq 1$.
By Lemma \ref{lemma:square}, we have
$$|\omega(g_n) |=|\omega(g_{n-1}) |^2.$$
Therefore,
\begin{equation}\label{equ:1}
|\omega(g_n) |=|\omega(g) |^{2^n}.
\end{equation}
Similarly, we have
$$|\omega(j(g_n)) |=|\omega(j(g)) |^{2^n}.$$

For any parabolic transformation $h \neq g_0$ of $\Gamma$, $g_0 \circ h$ is hyperbolic (one can check that it has two fixed points) and
\begin{equation}\label{equ:2}
tr(g_0\circ h)=|2+ \omega(h) |.
\end{equation}
We apply $(\ref{equ:2} )$ to $g_n$ and $j(g_n)$. Then by Lemma \ref{lemma:tr},
$$|2+\omega(f(g_n)) | \leq |2+\omega(g_n) |^s.$$
By $(\ref{equ:1} )$, we have
\begin{eqnarray*}
(|\omega(j(g)) |^{2^n}-2)&=& (|\omega(j(g_n)) |-2)\\
&\leq&| 2+\omega(j(g_n))|\\
&\leq&| 2+\omega(g_n)|^s\\
&\leq& (2+|\omega(g)|^{2^n})^s.\\
\end{eqnarray*}
Hence
$$(|\omega(j(g)) |^{2^n}-2)^{1/ 2^n} \leq (2+|\omega(g)|^{2^n})^{s/2^n}.$$
By letting $n\to \infty $, we obtain $|\omega(j(g)) |\leq |\omega(g)|^s$.
\end{proof}

The following lemma is also due to Sorvali (\cite{Sorvali74}  p. 3).

\begin{lemma} \label{lemma:lim}
Suppose $g\in \Gamma$ is hyperbolic and $\lambda(j(g))=\lambda(g)^a$. For each $n=1, 2, \cdots$, let $b_n$ be the real number such that
$$|\omega(j(g)^n \circ g_0 \circ j(g)^{-n})|=|\omega(g^n \circ g_0 \circ g^{-n})|^{b_n}.$$
Then $\lim_{n\to \infty} b_n = a$.
\end{lemma}

 \begin{proof}
The proof is due to Sorvali \cite{Sorvali73}. We include it here for completeness.

Note that the translation vectors of parabolic elements are not changed under conjugation by
translations $z \mapsto z + b$ (where $b$ is real). Therefore we may assume that $P(g)=P(j(g))=0$.

Set $\lambda=\lambda(g)$ and $N=N(g)$; then
$$g^n(z)= \frac{Nz}{(1-\lambda^n)z+ \lambda^nN},$$
and it is easy to show that
$$g^n \circ g_0 \circ g^{-n}(z)=\frac{N(\lambda^nN+\lambda^n-1)z+ N^2}{-(\lambda^2-1)z+ N(\lambda^nN+\lambda^n-1)}.$$
Hence
\begin{equation}\label{equ:3}
\omega(g^n \circ g_0 \circ g^{-n})=-\frac{(\lambda^n-1)^2}{\lambda^nN^2}=-\frac{\lambda^n+ \lambda^{-n}-2}{N^2}.
\end{equation}
By replacing $N$ by $N'=N(j(g))$ and $\lambda$ by $\lambda^a=\lambda(j(g))$, we obtain a similar expression for $\omega(j(g)^n \circ g_0 \circ j(g)^{-n})$. As a result,
$$\left(  \frac{\lambda^{an}+ \lambda^{-an}-2}{(N')^2}  \right)^{1/n}=\left(  \frac{\lambda^{n}+ \lambda^{-n}-2}{(N)^2}  \right)^{b_n/n}.$$

The left hand side of the above equation tends to $\lambda^a$ as $n \to \infty$, and
$$\left(  \frac{\lambda^{n}+ \lambda^{-n}-2}{(N)^2}  \right)^{1/n}$$
tends to $k$ as $n\to \infty $. It follows that $\lim_{n\to \infty} b_n = a $.
\end{proof}
It would be interesting to give a  geometric interpretation of Lemma \ref{lemma:lim}.

Now suppose that $j:\Gamma \to \Gamma'$ is an isomorphism with $\delta_L(j)< \infty$.
We let
$$\rho_L(j) = \inf\{a\geq 1 : |\omega(j(g))|\leq  |\omega(g)|^a \hbox{ for all parabolic $g$ in $\Gamma$}\}.$$
\begin{theorem} \label{thm:parabolic}  Let $\Gamma$ and $\Gamma'$ be two Fuchsian groups, both of which contains the parabolic element $g_0(z)=z+1$. Suppose that $j:\Gamma \to \Gamma'$ is an isomorphism with $j(g_0)=g_0$ and $\delta_L(j)< \infty$. Then $\delta_L(j)=\rho_L(j)$.
\end{theorem}
\begin{proof}
By Lemma \ref{lemma:para}, $\rho_L(j)\leq \delta_L(j)$.

Suppose that $\rho_L(j)< \delta_L(j)$. By definition of $\delta_L(j)$, there exists some hyperbolic elements $g \in \Gamma$, some number $\epsilon >0$ and $ a\geq \rho_L(j)+\epsilon$ such that $\lambda(j(g))=\lambda(g)^a$. By Lemma \ref{lemma:lim}, the numbers $b_n$ be numbers such that
$$|\omega(j(g)^n \circ g_0 \circ j(g)^{-n})|=|\omega(g^n \circ g_0 \circ g^{-n})|^{b_n}.$$
Then $\lim_{n\to \infty} b_n = a$.
This means that $\rho_L(j)\geq \rho_L(j)+\epsilon/2$, which is impossible.
As a result, $\rho_L(j) \geq  \delta_L(j)$.

\end{proof}

\section{Thurston's metric}\label{Lip}
Let $S=S_{g, n}$  with $n>0$. Each hyperbolic structure $X$ on $S$ can be represented by $\mathbb{H}/ \Gamma$ for some Fuchsian group $\Gamma$. Up to conjugation, we may assume that $g_0: z \mapsto z+1$ belongs to $\Gamma$, and we shall do this throughout the rest of this paper.

Given two hyperbolic structures  $X=\mathbb{H}/ \Gamma$ and $Y=\mathbb{H}/ \Gamma'$ on $S$, the identity map between  $(S,X)$ and $(S,Y)$ lifts to a homeomorphism $f: \mathbb{H} \to \mathbb{H}$ which extends continuously to the ideal boundary $\overline{\mathbb{R}}$. We may also assume that $f$ fixes $0, 1, \infty$. Using the map $f$, we define an isomorphism $j: \Gamma \to \Gamma'$ by
$$j(g):= f \circ g \circ f^{-1},  \forall \ g\in \Gamma. $$

Note that $j$ satisfies the assumptions on the isomorphism $j$ of Section \ref{iso}, that is,  $j$ is type-preserving and it fixes the element $g_0$ (which, also by assumption, is in both groups $\Gamma$ and $\Gamma'$).

Each hyperbolic element $g$ in $\Gamma$ corresponds to a (not necessary simple) closed geodesic $\gamma_g$ of $X$, with hyperbolic length $\ell_X(\gamma_g)=\log \lambda(g)$.
By  Theorem \ref{thm:parabolic}, we have
$$\sup_{g  \in \Gamma} \frac{\ell_Y(\gamma_g)}{\ell_X(\gamma_g)}=  \rho(j).$$

By a result of Thurston (Proposition 3.5,  \cite{Thurston}),
$$\sup_{g  \in \Gamma} \frac{\ell_Y(\gamma_g)}{\ell_X(\gamma_g)}=K(X,Y).$$

As a result, we get a new formula for the Thurston distance, which we state in the following:
\begin{theorem} \label{thm:lip}
With the above notation, Thurston's metric is given by
$$d_L(X,Y)=\log \delta_L(j)=\log \rho_L(j).$$
\end{theorem}
The following theorem of Sorvali  is a direct corollary of Theorem \ref{thm:lip}.
\begin{theorem}[Sorvali \cite{Sorvali74}] \label{thm:ls}
With the above notation, the length-spectrum metric satisfies
$$d_{ls}(X,Y)= \log\delta(j)=\log \rho(j),$$
where
$$  \delta(j)=\max \{\delta_L(j), \delta_L(j^{-1})\}$$
and
$$  \rho(j)=\max \{  \rho_L(j),  \rho_L(j^{-1})\}. $$
\end{theorem}

Note that it is interesting to have, like above, a definition of the Thurston metric in terms of translation lengths in the setting of groups, because such a definition can be generalized to other group actions. At the end of this paper, we address some open questions in this direction.

\section{Boundary mappings}\label{boundary}
A homeomorphism $\varphi: \overline{\mathbb{R}}\to \overline{\mathbb{R}}$ is called a \emph{boundary mapping} of an isomorphism
$j: \Gamma \to \Gamma'$ if
$$\varphi\circ g= j(g)\circ \varphi$$
for all $g \in \Gamma$. We say that $j$ is induced from $\varphi$.

The \emph{limit set} of a Fuchsian group $\Gamma$ , denoted by $\Lambda(\Gamma)$, is the accumulation points of the set
$$\Gamma(z_0)=\{ g(z_0) : g\in \Gamma \}$$
on $\overline{\mathbb{H}}={\mathbb{H}}\cup \mathbb{R}$ for some $z_0\in \mathbb{H}$. Since $\Gamma$ acts properly discontinuously and freely on $\mathbb{H}$, $\Lambda(\Gamma)$ is a subset of $\overline{\mathbb{R}}$. It is known that the definition of  $\Lambda(\Gamma)$ is independent of the choice of $z_0\in \mathbb{H}$. We say that a Fuchsian groups $\Gamma$ is \emph{non-elementary} if  $\Lambda(\Gamma)$ contains at least three points.

We denote by $F(\Gamma)$ the set of hyperbolic fixed points of a Fuchsian group $\Gamma$.

\begin{lemma} \label{lem:limit}
Suppose that $\Gamma$ is non-elementary. Then $\Gamma$ contain a hyperbolic element. Moreover, $F(\Gamma)$ is dense in $\Lambda(\Gamma)$.
\end{lemma}
See e.g. Matsuzaki-Taniguchi \cite{MT} for a proof of Lemma  \ref{lem:limit}. A Fuchsian group $\Gamma$ is \emph{of the first kind} if $\Lambda(\Gamma)=\overline{\mathbb{R}}$.

It is known that if  $\Gamma$ and $\Gamma'$ are finitely generated and of the first kind, then any type-preserving isomorphism $j: \Gamma \to \Gamma'$ is realized by some homeomorphism $\varphi: \overline{\mathbb{R}}\to \overline{\mathbb{R}}$. Moreover, the existence of such an isomorphism $j$ implies that $\Gamma$ and $\Gamma'$ are quasiconformally conjugate, that is, there exists a quasiconformal map $\Phi:  {\mathbb{H}}\to {\mathbb{H}}$ such that
$$\Phi\circ g= j(g)\circ \Phi$$
for all $g \in \Gamma$.

If an isomorphism $j: \Gamma \to \Gamma'$ is realized by some homeomorphism $\varphi: \overline{\mathbb{R}}\to \overline{\mathbb{R}}$, then it follows from work of Douady-Earle \cite{DE}  that there is a homeomorphism $\Phi:  {\mathbb{H}}\to {\mathbb{H}}$ such that  $\Phi|_{\overline{\mathbb{R}}}=\varphi$ and
$$\Phi\circ g= j(g)\circ \Phi$$
for all $g \in \Gamma$. Furthermore, $\varphi$ is quasisymmetric if and only if $\Phi$ is quasiconformal.
Recall that an orientation-preserving homeomorphism $\varphi: \overline{\mathbb{R}}\to \overline{\mathbb{R}}$, normalized by  $\varphi(\infty)=\infty$, is \emph{quasisymmetric} if
there exists a real number $k\geq 1$ such that for all $x, t \in \mathbb{R}, t\neq 0 $,
$$1/k \leq \frac{\varphi(x+t)-\varphi(x)}{\varphi(x)-\varphi(x-t)}\leq k.$$

In general, for an isomorphism $j: \Gamma \to \Gamma'$ between two Fuchsian groups of the first kind, if there exists a boundary mapping $\varphi$ of $j$, then $\varphi$ is unique. The following necessary and sufficient condition for existence of boundary mappings
is due to Sorvali \cite{Sorvali72}.

\begin{theorem} \label{th:S} Let $\Gamma$ and  $\Gamma'$ be two Fuchsian groups of the first kind and
$j: \Gamma \to \Gamma'$ an isomorphism. We do not make the assumption that $j$ is type-preserving. Then  the following
two conditions are equivalent:
\begin{enumerate}
\item The boundary mapping $\varphi$ of $j$ exists.
\item For all $g_1, g_2\in \Gamma$ not equal to the identity, $A(g_1)\cap A(g_2)\neq \emptyset$ if and only if $A(j(g_1))\cap A(j(g_2))\neq \emptyset$. Here $A(g)$ denotes the geodesic connecting $P(g)$ and $N(g)$ if $g$ is hyperbolic and denotes $P(g)$ if $g$ is parabolic.
\end{enumerate}
\end{theorem}
An example of a  type-preserving isomorphism between two Fuchsian  groups of the first kind whose boundary mapping does not exist is  given by Sorvali \cite{Sorvali72}.

\bigskip

Given $0<\alpha\leq 1$ and a subset $F \subset \overline{\mathbb{R}}$. We say that a homeomorphism $\varphi: \overline{\mathbb{R}}\to \overline{\mathbb{R}}$ is   \emph{$\alpha$-H\"older bi-continuous on $F$} if for each $x_0\in F$, there exists a neighborhood $I \subset \overline{\mathbb{R}}$ of $x_0$ and a constant $C\geq 1$ such that
$$ \frac{|x-x_0|^{1/\alpha}}{C} \leq |\varphi(x)-\varphi(x_0)|\leq C |x-x_0|^\alpha$$
for all $x\in I$. Note that if $x_0=\infty$ or $\varphi(x_0)=\infty$, then we consider the H\"older bi-continuity of $\varphi(1/x)$ at  $0$ or of $1/\varphi(x)$ at $x_0$ respectively.

Suppose that $\varphi$ is a \emph{boundary mapping} of an isomorphism
$j: \Gamma \to \Gamma'$ with $\delta(j)<\infty$. Let $B(j)$ be the set of real numbers $\alpha$, $0<\alpha\leq 1$, such that $\varphi$ is $\alpha$-H\"older bi-continuous on $F(\Gamma)$.
Note that $B(j)$ is an interval contained in $[0,1]$.

 The following theorem is also due to Sorvali \cite{Sorvali73}.
\begin{theorem} \label{thm:holder}
Suppose that $\varphi$ is a \emph{boundary mapping} of an isomorphism
$j: \Gamma \to \Gamma'$ with $\delta(j)<\infty$. Then $B(j) \neq \emptyset$ and $$\delta(j)=\min_{\alpha \in B(j)} \{  \frac{1}{\alpha}\}.$$
\end{theorem}

In the rest of this section, $S$ is a connected orientable surface of infinite type and $R=\mathbb{H}/\Gamma_0$ is a hyperbolic structure on $S$.  We assume that $\Lambda(\Gamma_0)=\overline{\mathbb{R}}$, i.e. $\Gamma_0$ is a Fuchsian group of the first kind.  Up to conjugation, we may assume that the transformation $z\mapsto \lambda z$ for some  $\lambda>1$ belongs to $\Gamma_0$ and that $1$ is a fixed point of some elements in $\Gamma_0$.

The \emph{length-spectrum Teichm\"uller space}  $\mathcal{T}_{ls}(R)$ is the space of homotopy classes of  hyperbolic surfaces $X$ homeomorphic to $R$, with
$$L(R, X)=\sup_{\gamma\in \mathcal{S} } \{\frac{\ell_X(\gamma)}{\ell_R(\gamma)}, \frac{\ell_R(\gamma)}{\ell_X(\gamma)}\} < \infty.$$

It is clear that for any two distinct elements $X,Y \in \mathcal{T}_{ls}(R)$, we have
$$1<L(X, Y)=\sup_{\gamma\in \mathcal{S} } \{\frac{\ell_X(\gamma)}{\ell_Y(\gamma)}, \frac{\ell_Y(\gamma)}{\ell_X(\gamma)}\} < \infty.$$
The length-spectrum distance between $X$ and $Y$ is then given by
$$d_{ls}(X,Y)=\frac{1}{2}\log L(X,Y).$$
The fact that $d_{ls}(X,Y)=0$ implies  $X=Y$ is due to Sorvali \cite{Sorvali73}. In fact, Sorvali \cite{Sorvali73} showed that this result is also valid for Fuchsian groups of the second kind with some
restriction on the isomorphism $j$.

Finally, we consider the \emph{quasiconformal Teichm\"uller space}  $\mathcal{T}_{qc}(R)$,
the space of homotopy classes of hyperbolic metrics $X$ on $R$ such that the
identity map between  the topological surface equipped respectively with $R$ and $X$ is homotopic to a quasiconformal homoemorphism.

For any two (equivalence classes of) hyperbolic metrics $X, Y\in \mathcal{T}_{qc}(R)$, their quasiconformal distance $d_{qc}(X,Y)$ is defined as
$$d_{qc}(X,Y)=\frac{1}{2} \log  \inf_{f}K(f)$$
where $K(f)$ is the quasiconformal dilatation of a quasiconformal homeomorphism
$f:X\to Y$ which is homotopic to the identity.

The following lemma is called Wolpert's formula \cite{Wolpert79}.
\begin{lemma}
For any $K$-quasiconformal map $f: X \to Y$ and any $\gamma\in \mathcal{S}$, we have
$$\frac{1}{K}\leq \frac{\ell_Y(f(\gamma))}{\ell_X(\gamma)}\leq K.$$
\end{lemma}

It follows from Wolpert's Lemma that for any two points $X, Y \in \mathcal{T}_{qc}(R)$,
\begin{equation}\label{equation:4}
d_{ls}(X,Y)\leq d_{qc}(X,Y).
\end{equation}
We note that the inequality $(\ref{equation:4})$ was first obtained by Sorvali \cite{Sorvali73}.

We proved in  \cite{ALPS} that if there exists a sequence of simple closed curves $\{\alpha_i\}$ contained in the interior of $R$ with $\ell_R(\alpha_i)\to 0$, then $\mathcal{T}_{qc}(R)\subsetneqq \mathcal{T}_{ls}(R)$.
The idea was to construct a sequence of hyperbolic metrics by performing large twists along short curves.

Suppose that $X=\mathbb{H}/\Gamma \in \mathcal{T}_{ls}(R)\setminus \mathcal{T}_{qc}(R)$.
The identity map between $(S,R)$ and $(S,X)$ lifts to a homeomorphism between the universal covers and induces a  homeomorphism $\varphi: \overline{\mathbb{R}}\to \overline{\mathbb{R}}$. Consider the homeomorphism
$j: \Gamma_0 \to \Gamma$ given by
$$\varphi\circ g= j(g)\circ \varphi$$
for all $g \in \Gamma_0$. Then $\varphi$ is a boundary mapping of $j$.
Since $\delta(j)< \infty$, it follows from Theorem \ref{thm:holder}  that there is some
$0<\alpha\leq 1$ such that $\varphi$ is $\alpha$-H\"older continuous on $F(\Gamma)$.
However, $\varphi$ is not quasisymmetric, since, otherwise,  by Douady-Earle \cite{DE} there would exist an quasiconformal extension of
$\varphi$ to $\mathbb{H}$ satisfying
$$\varphi\circ g= j(g)\circ \varphi,$$
and this would induce a quasiconformal mapping between $R$ and $X$. This is impossible, since $d_{qc}(R,X)=\infty$.

In a recent paper \cite{ALPS2}, we proved that endowed with the length-spectrum metric,
$\mathcal{T}_{qc}(R)$ is nowhere dense in $\mathcal{T}_{ls}(R)$.

Denote by $A(\Gamma_0)$  the set of orientation-preserving homeomorphisms $\varphi: \overline{\mathbb{R}}\to \overline{\mathbb{R}}$ such that
\begin{enumerate}
\item  The conjugation
$$\varphi\circ g \circ \varphi^{-1},  \  g\in \Gamma_0$$ gives a isomorphism between $\Gamma_0$ and some Fuchsian group $\Gamma$;
\item the map $\varphi$ fixes $0, 1, \infty$;
\item there exists some
$0<\alpha\leq 1$ such that $\varphi$ is $\alpha$-H\"older bi-continuous on $F(\Gamma_0)$.
\end{enumerate}
Note that if $\varphi: \overline{\mathbb{R}}\to \overline{\mathbb{R}}$ is quasisymmetric, then it is  $\alpha$-H\"older bi-continuous on $\overline{\mathbb{R}}$  for some $0<\alpha\leq 1$ (see e. g. \cite{LV}).

 Let $A_{qs}(\Gamma_0)$  be the subset of $A(\Gamma)$ consisting of the maps $\varphi \in A(\Gamma_0)$
which are quasisymmetric.
We conclude with the following theorem.
\begin{theorem} \label{thm:qs}
$A_{qs}(\Gamma)$ is a proper subset of $A(\Gamma)$.
\end{theorem}

\section{Cross-ratio norm}\label{cr}
For any four-tuple of distinct points $p,q,r,s$ on  $\overline{\mathbb{R}}$, the cross-ratio $(p, q, r,s)$ is defined by
$$(p, q, r,s)=\frac{p-r}{p-s}\cdot \frac{q-s}{q-r}.$$
Given a hyperbolic transformation $g$,  it is easy to show that
$$\lambda(g)=(g(s), s, N(g), P(g))$$
for any $s\neq P(g), N(g)$.

Suppose that there is an isomorphism $j: \Gamma \to \Gamma'$ with boundary mapping $\varphi$. Then for any $g\in \Gamma$, we have
\begin{eqnarray*}
\lambda(j(g))&=&(j(g)(t), t, N(j(g)), P(j(g))) \\
&=&   (\varphi\circ g (s), \varphi(s), \varphi( N(g)), \varphi(P(g)))   \\
\end{eqnarray*}
for any $s\neq P(g), N(g)$ and $t\neq P(j(g)), N(j(g))$.
Define $\|\varphi\|_{ls}=\delta(j)$.
It follows that
$$\|\varphi\|_{ls}=\sup \frac{|\log  (\varphi\circ g(s), \varphi(s), \varphi( N(g))|, \varphi(P(g)))}{(g(s), s, N(g), P(g))}$$
where the supremum is taken over all $g\in \Gamma$ and $s\neq P(g), N(g)$.

There is a natural norm on the set of orientation-preserving homeomorphisms of $\overline{\mathbb{R}}$,
called \emph{cross-ratio norm}, defined by
$$\|\varphi\|_{cr}=\sup \frac{|\log  (\varphi (p), \varphi(q), \varphi(r)|, \varphi(s))}{(p, q, r, s)}.$$
where the supremum is  taken over all four-tuples $(p, q, r, s)$ arranged in counter-clockwise  on  $\overline{\mathbb{R}}$.  It is clear that for a boundary mapping $\varphi$ of some isomorphism $j$, $\|\varphi\|_{ls}\leq \|\varphi\|_{cr}$.

The cross-ratio norm was studied by Gardiner-Hu-Lakic \cite{GHL}  and Hu
\cite{Hu}.  These authors proved that for an orientation-preserving homeomorphism $\varphi: \overline{\mathbb{R}} \to \overline{\mathbb{R}}$,  $\|\varphi\|_{cr}$ is equivalent to Thurston's norm on
the transverse shearing measure induced by the earthquake map on $\mathbb{H}$ whose extension to $\overline{\mathbb{R}}$ is equal to $\varphi$. $\|\varphi\|_{cr}$ is finite if and only if $\varphi$ is quasisymmetric.

We end this paper with the following questions: 
\begin{enumerate}
\item There are some sufficient conditions on $\Gamma$ such that $A_{qs}(\Gamma)=A(\Gamma)$ (see \cite{ALPS}, \cite{ALPS2}). Find sufficient and necessary conditions such that  $A_{qs}(\Gamma)=A(\Gamma)$.
\item Adapt this theory to the setting of automorphism groups of free groups.
\end{enumerate}

\end{document}